\documentclass[a4paper,11pt]{amsart}
\usepackage{amsmath}
\usepackage{amssymb}
\usepackage{amsthm}
\usepackage[lmargin=2.5 cm,rmargin=2.5 cm,tmargin=3.5cm,bmargin=2.5cm,
paper=a4paper]{geometry}
\newcommand{\md}{{\rm d}}

\newtheorem{thm}{Theorem}

\newtheorem{lem}[thm]{Lemma}

\newtheorem{rem}[thm]{Remark}

\newcommand{\R}{\mathbb R}
\newcommand{\C}{\mathbb C}

\title[Ginzburg-Landau equations]{A non-existence result for the Ginzburg-Landau equations}

\author[A. Kachmar]{Ayman Kachmar}
\author[M. Persson]{Mikael Persson}
\address{A. Kachmar \& M. Persson\newline
Aarhus University\\ Department of Mathematical Sciences \\Ny
Munkegade 8000 Aarhus C}
\email{akachmar@imf.au.dk}
\email{mickep@imf.au.dk}
\begin{document}
\begin{abstract}
We consider the stationary Ginzburg-Landau equations in $\R^d$, $d=2,3$. We
exhibit a class of   applied magnetic fields (including constant
fields) such that the Ginzburg-Landau equations do not admit finite
energy solutions.
\par \vspace*{10pt} \noindent {\sc R\'esum\'e.} {\bf Un r\'esultat
de non-existence pour les \'equations de Ginzburg-Landau.} Nous
consid\'erons les \'equations de Ginzburg-Landau dans $\R^d$,
$d=2,3$. Nous exhibons une classe de champs magn\'etiques
appliqu\'es telle que les \'equations de Ginzburg-Landau n'admettent
pas de solution d'\'energie finie.
\end{abstract}

\maketitle

\section{Introduction}
The aim of the present note is to study the Ginzburg-Landau system
of equations in $\R^2$,
\begin{equation}\label{GL-sys}
\left\{
\begin{array}{l}
-(\nabla-iA)^2\psi=(1-|\psi|^2)\psi\,,\\
-\nabla^\bot\left({\rm curl}\,A- H\right)={\rm
Im}\left(\psi\,\overline{(\nabla-iA)\psi}\right)\,.
\end{array}\right.
\end{equation} Here $\psi\in H^1_{\rm loc}(\R^2;\mathbb C)$ is
the complex order parameter, $A\in H^1_{\rm loc}(\R^2;\R^2)$ is the
magnetic vector potential, ${\rm curl}\,A$ is the induced magnetic
field
\begin{equation}\label{eq:mf}
B={\rm curl}\,A=
\partial_{x_1}A_2-\partial_{x_1}A_1\,,\end{equation} $H\in L_{\rm
loc}^2(\R^2)$ is the applied magnetic field, and
$\nabla^\bot=(-\partial_{x_2},\partial_{x_1})$ is the Hodge
gradient.

Solutions of \eqref{GL-sys} are of particular interest in the
physics literature  as  they do include periodic solutions with
vortices distributed in a uniform lattice, named as Abrikosov's
solution. We refer the reader to \cite{Ab} for the physical
motivation and to \cite{AS, D} for mathematical results in that
direction.

The equations \eqref{GL-sys} are formally the Euler-Lagrange
equations of the following Ginzburg-Landau energy,
\begin{equation}\label{eq-en-GL}
\mathcal
G(\psi,A)=\int_{\R^2}\Bigl(|(\nabla-iA)\psi|^2+\frac12\bigl(1-|\psi|^2\bigr)^2
+|{\rm curl}\,A-H|^2\Bigr)\,\md x\,.
\end{equation}
A solution $(\psi,A)$ of \eqref{GL-sys} is said to have finite
energy if $\mathcal G(\psi,A)<\infty$.  When the applied magnetic
field $H\in L^2(\R^2)$, it is proved in \cite{GS, Y1} that the
system \eqref{GL-sys} admits finite energy solutions. In the present
note, we would like to discuss the optimality of the hypothesis
$H\in L^2(\R^2)$ thereby establishing negative results when this
hypothesis is violated.

Our result is that if $H$ is not allowed to decay  fast at infinity
(especially if it is constant), then there are no finite energy
solutions to~\eqref{GL-sys}.

\begin{thm}\label{cor:GL-sys}
Let $\alpha<1$. Assume that the applied magnetic field
$H\in L^2_{\rm loc}(\R^2)$ and that there exist constants $R_0>0$ and $h>0$
such that
\begin{equation*}
H(x)\geq \frac{h}{|x|^\alpha}\quad {\rm for ~all~}x {\rm ~with~}|x|>R_0\,.
\end{equation*}
 Then the Ginzburg-Landau system \eqref{GL-sys} does not admit
finite energy solutions.
\end{thm}

\begin{rem}
We note that $\frac1{|x|^\alpha}\in L^2(\R^2\setminus B(0,1))$ if and only if
$\alpha>1$, which
means that the result in Theorem~\ref{cor:GL-sys} is really complementary
to the results in~\cite{GS, Y1}.
\end{rem}

\begin{rem}\label{rem:alternative}
The same non-existing result still holds if we instead impose the following 
properties on $H$: $(1)$ $H\not\in L^2(\R^2)$, $(2)$ There exists $R_0>0$ 
such that for $H(x)$ is 
positive for $|x|>R_0$, and $(3)$ there exists $R_1>0$ such that the 
reverse H\"older-inequality
\begin{equation}\label{eq:alternative}
\int_{B(0,R)} H(x) \,\md x \geq |B(0,R)|^{1/2} \biggl(\int_{B(0,R)} H(x)^2 \,\md x\biggr)^{1/2}
\end{equation}
holds for all $R>R_1$. The proof follows the proof of Theorem~\ref{cor:GL-sys}
until the end, where the alternative properties of $H$ are used.
\end{rem}

We conclude by mentioning an immediate generalization to the 3-dimensional
equations. Let $\mathbf H=(H_1,H_2,H_3)\in L^2_{\rm loc}(\R^3;\R^3)$
be a given vector field. Consider the Ginzburg-Landau equations in $\R^3$,
\begin{equation}\label{GL-3}
\left\{
\begin{array}{l}
-(\nabla-iA)^2\psi=(1-|\psi|^2)\psi\,,\\
-{\rm curl}\left({\rm curl}\,A- H\right)={\rm
Im}\bigl(\psi\,\overline{(\nabla-iA)\psi}\bigr)\,.
\end{array}\right.
\end{equation}
A solution $(\psi,A)\in H^1_{\rm loc}(\R^3;\C)\times H^1_{\rm loc}(\R^3;R^3)$ 
is said to have finite energy if
\begin{equation*}
\mathcal E(\psi,A)=
\int_{\R^3}\Bigl(|(\nabla-iA)\psi|^2+\frac12(1-|\psi|^2)^2+|{\rm
curl}\,A-H|^2\Bigr)\,\md x<\infty\,.
\end{equation*}
We have then a similar result to Theorem~\ref{cor:GL-sys}.
\begin{thm}\label{cor:GL-sys3}
Let $\alpha<\frac32$. Assume that there exist $h>0$ and $R_0>0$ such that
 the applied magnetic field
$\mathbf H=(H_1,H_2,H_3)\in L^{2}_{\rm loc}(\R^3;\R^3)$ satisfies,
\begin{equation*}
H_3(x)\geq \frac{h}{|x|^\alpha}\quad\forall~x {\rm ~such~that~}\ |x|\geq R_0\,.
\end{equation*}
Then the Ginzburg-Landau system \eqref{GL-3} does not admit finite energy solutions.
\end{thm}

\begin{rem}
The Remark~\ref{rem:alternative} carries over to three dimensions, but with any 
component $H_j$ in place of $H$.
\end{rem}

The proof of Theorem~\ref{cor:GL-sys3} is exactly the same as that of 
Theorem~\ref{cor:GL-sys}. So,
we will give details only for the proof of Theorem~\ref{cor:GL-sys}. The
essential key for proving Theorem~\ref{cor:GL-sys} is a result from the
spectral theory of magnetic Schr\"odinger operators stated in Lemma~\ref{lem:sh-op}
below.

\section{Two auxiliary lemmas}

We start with the following  observation concerning the Ginzburg-Landau
system \eqref{GL-sys}.

\begin{lem}\label{lem:GL}
Assume that $H\in L_{\rm loc}^2(\R^2)$. Let $(\psi,A)$ be a weak
solution of \eqref{GL-sys} such that $\mathcal G(\psi,A)<\infty$.
Then $|\psi|\leq 1$ in $\R^2$.
\end{lem}
\begin{proof}
This result was proved by Yang \cite[Lemma~3.1]{Y} for
$\mathbb{R}^3$ under the assumption $H\in L^2(\R^3)$. The assumption
on $H$ is not used in Yang's proof but the proof only relies on  the
fact that the energy of $(\psi,A)$ is finite. The proof of this
lemma is line-by-line the same as \cite{Y}, but with $\mathbb{R}^2$
in place of $\mathbb{R}^3$.
\end{proof}

A key-ingredient is the following result from the
spectral theory of magnetic Schr\"odinger operators.


Let $\chi$ be a cut-off function such that $0\leq \chi\leq 1$,
$\chi=1$ in $[0,\frac12]$ and $\chi=0$ in $[1,\infty)$. For all
$R>0$, we introduce the function,
\begin{equation}\label{eq-chi-R}
\chi_R(x)=\chi\left(\frac{|x|}{R}\right)\quad\forall~x\in\mathbb
R^2\,.
\end{equation}

\begin{lem}\label{lem:sh-op}
There exists a constant $C>0$ such that, for all $\psi\in
H^1(\mathbb R^2;\mathbb C)$, $A\in H^1_{\rm loc}(\R^2;\R^2)$ and
$R>0$, the following inequality holds,
\begin{equation*}
\int_{B(0,R)}|(\nabla-i A)\psi|^2\,\md x\geq
\frac12\int_{B(0,R)}B(x)|\chi_R\psi|^2\,\md
x-\frac{C}{R^2}\int_{B(0,R)\setminus B(0,R/2)}|\psi(x)|^2\,\md x\,.
\end{equation*}
Here $B={\rm curl}\,A$ and $\chi_R$ the function from
\eqref{eq-chi-R}.
\end{lem}
\begin{proof}
We write,
\begin{eqnarray*}
\int_{B(0,R)}|(\nabla-iA)\psi|^2\,\md x&\geq&
\int_{B(0,R)}|\chi_R(\nabla-iA)\psi|^2\,\md x\\
&\geq&\frac12\int_{B(0,R)}|(\nabla-iA)(\chi_R\psi)|^2\,\md
x-\int_{B(0,R)}|\psi\nabla\chi_R|^2\,\md x\,.
\end{eqnarray*}
To finish the proof, we just use the following well known inequality
(see \cite{CFKS} or \cite[Lemma~2.4.1]{FH}),
\begin{equation*}
\int_{B(0,R)}|(\nabla-iA)\phi|^2\,\md x\geq
\pm\int_{B(0,R)}B(x)|\phi|^2\,\md x\,,\quad\forall~\phi\in
H^1_0(B(0,R))\,.
\end{equation*}
\end{proof}

\section{Proof of Theorem~\ref{cor:GL-sys}}
Assume that $(\psi,A)$ is a finite energy solution of
\eqref{GL-sys}. Thanks to Lemma~\ref{lem:GL}  we have $|\psi|\leq1$
in $\R^2$.

Recalling the hypothesis on the applied magnetic field $H$ that we assumed in
Theorem~\ref{cor:GL-sys}, we may pick $R_0>0$ such that
\begin{equation}\label{eq:6}
H(x)\geq \frac{h}{|x|^\alpha}\quad\forall~|x|\geq R_0\,.
\end{equation}
Applying Lemma~\ref{lem:sh-op}, with $(\psi,A)$ as above, a solution
of (\ref{GL-sys}), we obtain with $B={\rm curl}\,A$,
\begin{equation*}
\int_{\mathbb
R^2}|(\nabla-iA)\psi|^2\,\md x\geq
\frac12\int_{B(0,R/2)}B(x)|\psi|^2\,\md x-\frac{C}{R^2}
\int_{B(0,R)\setminus B(0,R/2)}|\psi|^2\,\md x\,.
\end{equation*}
Let $R>2R_0$ and $\Omega_R=\{x\in\R^2~:~R_0<|x|<R\}$. Then we may
write,
\begin{eqnarray*}
\int_{\mathbb R^2}|(\nabla-iA)\psi|^2\,\md x&\geq&
\frac12\int_{\Omega_R}B(x)\,|\chi_R\psi|^2\,\md
x\\&&+\frac{1}{2}\int_{B(0,R_0)}B(x)|\chi_R\psi|^2\,\md
x-\frac{C}{R^2} \int_{B(0,R)\setminus B(0,R/2)}|\psi|^2\,\md x\,.
\end{eqnarray*}
Using  that $\int_{\mathbb R^2}|(\nabla-iA)\psi|^2\,\md x\leq
\mathcal{G}(\psi,A)$, $A\in H^1_{\rm loc}(\R^2)$ and
$|\chi_R\psi|\leq1$, we get a constant $C_0$ depending on $R_0$ such
that,
\begin{equation}\label{eq:s1'}
\mathcal{G}(\psi,A)\geq \frac12
\int_{\Omega_R}B(x)|\chi_R\psi|^2\,\md x-C_0\,.\end{equation} So,
let us handle the first term in the right hand side above. We write,
\begin{equation}\label{eq-ref1}
\int_{\Omega_R}B(x)|\chi_R\psi|^2\,\md x = \int_{\Omega_R}
H(x)|\chi_R\psi|^2\,\md x +\int_{\Omega_R}\bigl(B(x)-H(x)\bigr)|\chi_R\psi|^2
\,\md x\,.
\end{equation}
In order to handle the last term on the right of \eqref{eq-ref1}, we
apply  a Cauchy-Schwarz inequality and use the fact that
$|\chi_R\psi|\leq1$. In this way  we get,
\begin{equation*}
\left|\int_{\Omega_R}\bigl(B(x)-H(x)\bigr)|\chi_R\psi|^2 \,\md x\right|\leq
\left(\int_{\Omega_R}|B(x)-H(x)|^2\,\md
x\right)^{1/2}\left(\int_{\Omega_R}\md
x\right)^{1/2}\leq \bigl(\mathcal G(\psi,A)\bigr)^{1/2}|\Omega_R|\,.
\end{equation*} 
Implementing this bound
together with \eqref{eq:6} in the right side of \eqref{eq:s1'}, we
get the following lower bound,
\begin{equation}\label{eq:6'}
\int_{\Omega_R}B(x)|\chi_R\psi|^2\,\md x \geq
\int_{\Omega_R}\frac{h}{|x|^\alpha}|\chi_R\psi|^2\,\md
x-\left(\mathcal G(\psi,A)\right)^{1/2}|\Omega_R|^{1/2}\,.
\end{equation}
We need only to bound from below
$\int_{\Omega_R}\frac{h}{|x|^\alpha}|\chi_R\psi|^2\,\md x$.
Actually, using that $\chi_R=1$ in $B(0,R/2)$ and  a Cauchy-Schwarz
inequality, we obtain,
\begin{equation*}
\begin{aligned}
\int_{\Omega_R}\frac{h}{|x|^\alpha}|\chi_R\psi|^2\,\md x
&\geq\int_{\Omega_{R/2}} \frac{h}{|x|^\alpha}|\psi|^2\,\md
x\\
&=\int_{\Omega_{R/2}} \frac{h}{|x|^\alpha}\,\md x
+\int_{\Omega_{R/2}}\frac{h}{|x|^\alpha}(|\psi|^2-1)\,\md x\\
&\geq \frac{2\pi h}{2-\alpha} \bigl((R/2)^{2-\alpha}-R_0^{2-\alpha}\bigr)\\
&~
-\bigl(\mathcal G(\psi,A)\bigr)^{1/2}h\Bigl(\frac{2\pi}{2-2\alpha}\Bigr)^{1/2}\bigl((R/2)^{2-2\alpha}-R_0^{2-2\alpha}\bigr)^{1/2}\,.
\end{aligned}\end{equation*}

Now we use the assumption that $\mathcal G(\psi,A)<\infty$. In this way, we
get by implementing the right-hand side above in~\eqref{eq:6'} and
then by substituting the resulting lower bound into \eqref{eq:s1'},
a constant $C$ such that,
\begin{equation}\label{eq:7}
\mathcal{G}(\psi,A)
\geq  \frac{2^{\alpha-1}\pi h}{2-\alpha} R^{2-\alpha}-CR^{1-\alpha}-CR-C\,.
\end{equation}
Making $R\to\infty$  and recalling that
$\alpha<1$, we get a contradiction to the assumption that the energy
$\mathcal G(\psi,A)$ is finite, thereby finishing the proof of Theorem~\ref{cor:GL-sys}.

\section*{Acknowledgements}
The authors wish to thank the anonymous referee for valuable
suggestions. AK is supported by a Starting Independent Researcher
grant by the ERC under the FP7. MP is supported by the Lundbeck
Foundation.

\end{document}